\documentclass[a4paper,12pt]{article}

\usepackage[utf8]{inputenc}

\usepackage{bbm}

\usepackage{enumitem}

\usepackage{amsmath,amsfonts,amssymb,amsthm}

\def\R{\mathbb{R}}

\def\T{\mathcal{T}}

\def \ml {\begin{pmatrix}} %bmatrix, vmatrix, Vmatrix, matrix, smallmatrix
\def \mr {\end{pmatrix}}

\def \feq#1 {\begin{center}\framebox{$\displaystyle #1$}\end{center}}

\def \e{\varepsilon}

\newcommand{\Fix}{{\rm Fix}}

\reversemarginpar

\author{M. ZAVIDOVIQUE\thanks{IMJ, Universit\' e Pierre et Marie Curie, Case 247,
4 place Jussieu, 
F-75252 Paris cedex 05 } }
\newtheorem{df}{Definition}[section]
\newtheorem{lm}[df]{Lemma}

\newtheorem{Th}[df]{Theorem}

\title{Fixed points of contractions approximating $1$-Lipschitz maps}

\begin{document}

\maketitle

\begin{abstract}
A $1$-Lipschitz map $f$ from a convex compact set to itself has fixed points. This consequence of Brouwer's or Schauder's fixed point theorem has more elementary proofs by approximating $f$  by $\lambda$-contractions, $f_\lambda$. We study the convergence of the fixed points of those contractions as they converge to $f$.
\end{abstract}

\section*{Introduction}
Given a compact metric space $X$ and a continuous function $c : X\times X \to \R$, consider the map $T$, from $C^0(X,\R)$ to itself, defined by 
$$\forall u\in C^0(X,\R), \ \ \forall x\in X, \quad Tu(x) = \min_{y\in X} \big(u(y) + c(y,x)\big).$$
It is easily checked that $T$ is $1$-Lipschitz for the sup norm $\|\cdot \|_\infty$, therefore if $\lambda\in (0,1)$ the function $T_\lambda : u \mapsto T(\lambda u)$ is a $\lambda$-contraction from $C^0(X,\R)$ to itself. It follows from the Banach fixed point theorem that there exists a unique function $u_\lambda$ verifying $T_\lambda (u_\lambda) = u_\lambda$.  The following theorem was proved in \cite{DFIZ} (see also \cite{DFIZ2} for related results):
\begin{Th}\label{DFIZ}
There exists a constant $\alpha\in \R$ such that, as $\lambda \to 1$, $u_\lambda + \frac{\alpha}{1-\lambda}$ uniformly converges to a function $u_1: X\to \R$ that verifies the following equation :
$$u_1 = Tu_1 + \alpha.$$
\end{Th}

The function $u_1$ is called a weak KAM solution following Fathi and the proof of this theorem uses heavily explicit formulas for the functions $u_\lambda$, Aubry-Mather theory and Fathi's weak KAM theory.

However, the method of approaching the $1$-Lipschitz map $T$ by contractions $T_\lambda$ in order to find fixed points of $T$ can be applied in more general settings. Therefore it is a natural question wether the previous theorem is a particular case of a very general phenomenon or if it relies exclusively on the very special setting.

It turns out the answer is both yes and no.

 On the positive side, we prove the following theorem:
\begin{Th}\label{fp1}
Let $(E,\| \cdot \|)$ be a normed real vector space and  $C\subset E$ be a convex compact subset such that $0\in C$. Let  $f : C\to C$ be a $1$-Lipschitz map. For any $\lambda \in (0,1)$ let us set $x_\lambda$ to be the unique fixed point of the contraction $f_\lambda : x\mapsto f(\lambda x)$, that is the only point of $C$ such that $f(\lambda x_\lambda) = x_\lambda$.

If $\| \cdot \|$ is $C^1$ on $E \setminus \{0\}$, then the family $(x_\lambda)_{\lambda\in (0,1)}$ converges as $\lambda \to 1$.
\end{Th}

On the negative side, we provide  in the appendix an explicit example of a map in $\R^2$ equipped with the norm $\|\cdot \|_1$ for which the family $(x_\lambda)$ does not converge.

Of course, Theorem \ref{DFIZ} is not a consequence of Theorem \ref{fp1} for many reasons, but the most fundamental one is that the sup norm $\|\cdot \|_\infty$ on $C^0(X,\R)$ is not $C^1$. Indeed, balls for this norm have many corners. 

One may wonder if a result can be obtained for $1$-Lipschitz maps that do not have fixed points. The following theorem was established in \cite{KN} (see the reference for more precise statements and other very interesting results).

\begin{Th}[Kohlberg, Neyman]
Let $E$ be a strictly convex and reflexive Banach space and $f : E\to E$ a $1$-Lipschitz map. For any $\lambda \in (0,1)$ let us set $x_\lambda$ to be the unique fixed point of the contraction $f_\lambda : x\mapsto f(\lambda x)$, then the family $(1-\lambda ) x_\lambda $ weakly converges as $\lambda \to 1$. Moreover, it strongly converges if the dual $E^*$ has Fr\' echet differentiable norm.
\end{Th} 

\subsection*{Organization of the paper} In the first part, we start with the case of a pre-Hilbertian norm, the results being more precise in this case and the proof making use of more classical results.

In the second part, we will prove Theorem \ref{fp1}.

In the last part, we will discuss a generalization where the compactness hypothesis is dropped, in infinite dimensional Hilbert spaces.

Finally, in the appendix, we develop two examples in $\R^2$, one for which the conclusions of Theorem \ref{fp1} are not verified and one where the set of fixed points of $f$ is not convex.

\subsection*{Acknowledgement}
The author warmly thanks Patrice Le Calvez for his help in finding the example in Appendix \ref{A}, and Albert Fathi and Gilles Godefroy for their insight on the infinite dimensional aspects and Julien Grivaux for  his contribution in the concluding remark. He also thanks S\' ebastien Gou\" ezel for bringing the reference \cite{KN} to his attention.

\section{The pre-Hilbertian  case}\label{euclide}

In this section, the norm  $\| \cdot \|$ is obtained from a scalar product $\langle \cdot , \cdot \rangle$. We keep the previous notations: $C\subset E$ is a convex compact subset such that $0\in C$ and $f : C\to C$ is a $1$-Lipschitz map. If $\lambda \in (0,1)$, $x_\lambda \in C$ is the only vector verifying the relation $f(\lambda x_\lambda) = x_\lambda$. Moreover, we set $y_\lambda = \lambda x_\lambda$. It is the only vector verifying that 
\begin{equation}\label{lambda}
\lambda f(y_\lambda) = y_\lambda.
\end{equation}
 It is clear that  the convergence of the family $(x_\lambda)_{\lambda \in (0,1)}$  is equivalent to that of $(y_\lambda)_{\lambda \in (0,1)}$, as $\lambda \to 1$, and that if they converge, they have the same limit. We start by three lemmas that shed light on the situation:
 
 \begin{lm}
 The set $\Fix(f)$ of fixed points of $f$ is a compact, nonempty and convex set.
 \end{lm}

\begin{proof}
It is clearly compact by continuity of $f$  and nonempty as any accumulation point of $(x_\lambda)_{\lambda \in (0,1)}$ as $\lambda \to 1$ is a fixed point of $f$. The fact that it is convex is true as soon as balls are strictly convex (which is the case for pre-Hilbertian spaces). Indeed, let $x$ and $y$ be fixed points of $f$ and $t\in (0,1)$. Setting $z = tx +(1-t)y$,  then the $1$-Lipschitz property implies that $f(z)$ must lie  in both closed balls $\overline{B(x , (1-t)\|y-x\|)}$ and  $\overline{B(y , t\|y-x\|)}$. As their intersection is $\{z\}$ we get $f(z)=z$.
\end{proof}
 The next lemma appears in \cite{Browder} and is crucial to this paper.
\begin{lm}\label{monotone}
The function $Id - f : C\to E$ is monotone meaning that 
$$\forall (x,y)\in E \times E,\quad \big\langle \big(x-f(x) \big) -\big(y-f(y)\big), x-y\big\rangle \geqslant 0.$$
\end{lm}

\begin{proof}
The proof follows from a simple computation:
\begin{multline*}
 \big\langle \big(x-f(x) \big) -\big(y-f(y)\big), x-y\big\rangle =\\
 = \|x-y\|^2 -\langle f(x)-f(y) , x-y \rangle \\
 \geqslant  \|x-y\|^2 - \|f(x)-f(y)\| \cdot \|x-y\| \geqslant 0,
\end{multline*}
where we used the Cauchy-Schwarz inequality first and then the hypothesis that $f$ is $1$-Lipschitz.
\end{proof}

This next lemma will be used in section \ref{Hilbert}. However, we state it now as we believe it is of independent interest.

\begin{lm}
The function $\lambda \mapsto \|y_\lambda\|$ is constant if $0\in \Fix(f)$ and increasing otherwise.
\end{lm}

\begin{proof}
If $0\in \Fix(f)$ then clearly, $y_\lambda = 0$ for all $\lambda$.

Otherwise, $0$ does not verify \eqref{lambda}, except for $\lambda = 0$. It follows that $y_\lambda \neq 0$ for $\lambda \neq 0$ and that the map $\lambda \mapsto y_\lambda$ is injective. Let $0<\lambda <\lambda' < 1$. We apply the previous lemma to $y_\lambda$ and $y_{\lambda'}$. Using \eqref{lambda}, we obtain that 
$$\big \langle \big(y_{\lambda'} - \frac{1}{\lambda'} y_{\lambda'}\big) - \big(y_{\lambda} - \frac{1}{\lambda} y_{\lambda}\big), y_{\lambda'}-y_{\lambda}  \big \rangle= \big(1-\frac{1}{\lambda'}\big)\|y_{\lambda'}-y_{\lambda} \|^2+\big(\frac{1}{\lambda} -\frac{1}{\lambda'}\big)\langle y_\lambda , y_{\lambda'} - y_{\lambda}\rangle\geqslant 0.$$
Hence, $\langle y_\lambda , y_{\lambda'} - y_{\lambda}\rangle > 0$. Similarly, 
$$\big\langle \big(y_{\lambda'} - \frac{1}{\lambda'} y_{\lambda'}\big) - \big(y_{\lambda} - \frac{1}{\lambda} y_{\lambda}\big), y_{\lambda'}-y_{\lambda}\big\rangle= \big(1-\frac{1}{\lambda}\big)\|y_{\lambda'}-y_{\lambda} \|^2+\big(\frac{1}{\lambda} -\frac{1}{\lambda'}\big)\langle y_{\lambda'} , y_{\lambda'} - y_{\lambda}\rangle\geqslant 0.$$
Hence $\langle y_{\lambda'} , y_{\lambda'} - y_{\lambda}\rangle > 0$. Finally, by summing those two inequalities, one gets
$$\langle y_{\lambda'}+y_\lambda , y_{\lambda'} - y_{\lambda}\rangle= \|y_{\lambda'}\|^2 - \|y_\lambda \|^2 > 0.$$

\end{proof}

We finally state and prove the theorem in this case.

\begin{Th}\label{fixe1}
If $x^*\in \Fix(f)$ is the orthogonal projection of $0$ on $\Fix(f)$, then $\lim\limits_{\lambda \to 1} y_\lambda = x^*$.
\end{Th}

\begin{proof}
Once again, we apply Lemma \ref{monotone} to $x=y_\lambda$ for some $\lambda \in (0,1)$ and $y\in \Fix(f)$. We obtain that $\big(1-\frac{1}{\lambda}\big)\langle y_\lambda , y_\lambda - y \rangle \geqslant 0$ and consequently $\langle y_\lambda , y_\lambda - y \rangle \leqslant 0$. If $\lambda_n\to 1$ is a sequence such that $(y_{\lambda_n})_{n\in \mathbb N}$ converges to some point $z$, then $z\in \Fix(f)$ and passing to the limit we find that 
$$\forall y\in \Fix(f),\quad  \langle z,   z-y\rangle \leqslant 0. $$
It follows that $z=x^*$ and by compactness of $C$, the result is proved.
\end{proof}

\section{The  general compact case}

\subsection{For $C^1$ norms}
In this section, we assume that $\| \cdot \|$ is a norm on $E$ that is $C^1$ on $E \setminus \{ 0 \}$. It follows (and both facts are actually equivalent) that for all $x\in E$ there exists a unique linear form $\ell_x$ such that $\|\ell_x\| = 1$ and $\ell_x(x) = \|x\|$. Note that we use the same notation for the norm on $E$ and the norm induced on its dual $E^*$. The linear form $\ell_x$ is the differential of $\|\cdot \|$ at $x$. Its kernel is the direction of the tangent hyperplane to the sphere of radius $\|x\|$ at $x$. It follows that the map $x\mapsto \ell_x$ is continuous. Let us finally stress that for all $x\neq 0$, $\ell_{-x} = -\ell_x$.

Let us now state and prove the main result of this section.

\begin{Th}\label{fp2}
Let $(E,\| \cdot \|)$ be a normed vector space such that $\| \cdot \|$ is $C^1$ on $E \setminus \{0\}$. Let $C\subset E$ be a convex compact subset such that $0\in C$. Let $f : C\to C$ be a $1$-Lipschitz map. For any $\lambda \in (0,1)$ let us set $y_\lambda$  the unique fixed point of the contraction $\lambda f$, that is the only point of $C$ such that $\lambda f(y_\lambda) = y_\lambda$.

Then the family $(y_\lambda)_{\lambda\in (0,1)}$ converges as $\lambda \to 1$.
\end{Th}

Before entering the proof, let us stress that with those hypotheses, $\Fix(f)$ is not necessarily convex. The reader will see strong resemblance with the proof of Theorem \ref{fixe1}. The characterization we find of the limit has however a less clear geometric interpretation.

\begin{proof}
If $0\in \Fix(f)$ then $y_\lambda = 0$ for all $\lambda\in (0,1)$ and the result is straightforward.

In the remaining case, $y_\lambda \neq 0$ and the map $\lambda \mapsto y_\lambda $ is injective.
The first step is to mimic the monotonicity of $Id-f$. Let $x\neq y$ we compute that
\begin{multline*}
\ell_{x-y} \big( (x-f(x) ) -(y-f(y))\big) \\
= \ell_{x-y}(x-y) - \ell_{x-y}\big(f(x)-f(y)\big) \\
\geqslant \|x-y\|  -\|f(x) - f(y)\| \geqslant 0,
\end{multline*}
where we have used that $\|\ell_{x-y}\| = 1$ and then that $f$ is $1$-Lipschitz.

We then specialize the previous inequality to $x=y_\lambda$ for some $\lambda \in (0,1)$ and $y\in \Fix(f)$ to obtain $\big(1-\frac{1}{\lambda}\big)\ell_{y_\lambda - y}(y_\lambda) \geqslant 0$ and finally
\begin{equation}\label{crucial}
\forall y\in \Fix(f) , \quad \ell_{y_\lambda - y}(y_\lambda) \leqslant 0.
\end{equation}
Taking the limit in the previous inequality, we find that if $\bar y$ is a limit of a sequence $(y_{\lambda_n})_{n\in \mathbb N}$ with $\lambda_n\to 1$ then
$$\forall y\in \Fix(f) \setminus\{\bar y\},\quad \ell_{\bar y- y}(\bar y) \leqslant 0.$$
To conclude, we notice that there can be at most one point verifying the previous property. Indeed, assume by contradiction that for some $z \in \Fix(f)$, 
$$\forall y\in \Fix(f) \setminus\{z\},\quad \ell_{z- y}(z) \leqslant 0.$$
We discover that 
$$ \ell_{\bar y-z}(\bar y) \leqslant 0,$$
$$\ell_{z- \bar y}(z) = -  \ell_{\bar y-z}(z) \leqslant 0,$$
and summing those two inequalities
$$ \ell_{\bar y-z}(\bar y-z) = -\|\bar y - z\| \leqslant 0.$$
This is a contradiction.
\end{proof}

\subsection{For Gateaux-differentiable norms}

The hypotheses of the previous theorem  can be relaxed. First let us recall some definitions.
\begin{df}
A function $F : E \to \R$ is said to be Gateaux-differentiable at $x\in E$ if there exists a linear form $L\in E^*$ such that 
$$\forall v\in E,\quad \lim_{t\to 0} \frac{1}{t} \big( F(x+tv)-F(x)\big) = L(v).$$ 
\end{df}
We then apply this definition to the norm:
\begin{df}
A norm $\| \cdot \|$ is said to be smooth if it is Gateaux-differentiable on $  E\setminus \{0\}$.
\end{df}
If $\|\cdot \|$ is differentiable, if $x\neq 0$, we still denote by $\ell_x$ the Gateaux-differential of $\| \cdot \|$ at $x$. It can be proved that $\ell_x\in E^*$ is the only continuous linear form such that $\| \ell_x\| = 1$ and $\ell_x(x) = \|x\|$.  

\begin{lm}\label{weak}
Let $(x_n)_{n\in \mathbb N}$ be a sequence of non zero vectors such that $x_n\to x$ and $\ell_{x_n} \rightharpoonup \ell$. Then $\ell = \ell_x$.
\end{lm}

\begin{proof}
By properties of weak limits, $\| \ell \| \leqslant \liminf \| \ell_{x_n}\| = 1$. Moreover, one computes that 
$$\|x_n\| = \ell_{x_n}(x_n-x)+\ell_{x_n}(x).$$
As   $|\ell_{x_n}(x_n-x)|\leqslant \|x-x_n\|$ and $\ell_{x_n}(x)\to \ell(x)$, it follows that $\|x\| = \lim \|x_n\| = \lim \ell_{x_n}(x_n) = \ell (x)$. Hence the result.
\end{proof}

\begin{Th}\label{fp3}
Let $(E,\| \cdot \|)$ be a normed vector space such that $\| \cdot \|$ is smooth. Let $C\subset E$ be a convex compact subset such that $0\in C$. Let $f : C\to C$ be a $1$-Lipschitz map. For any $\lambda \in (0,1)$ let us set $y_\lambda$  the unique fixed point of the contraction $\lambda f$, that is the only point of $C$ such that $\lambda f(y_\lambda) = y_\lambda$.

Then the family $(y_\lambda)_{\lambda\in (0,1)}$ converges as $\lambda \to 1$.
\end{Th}

\begin{proof}
Let us give the elements to adapt the proof of Theorem \ref{fp2}. Up to restricting to $\overline{Vect(\{x\in C\})}$, we assume that $E$ is separable. It follows from a theorem of Banach that any bounded sequence in $E^* $ has a weakly converging subsequence.

Let $\lambda_n\to 1$ be such that $y_{\lambda_n} \to \bar y \in \Fix (f)$. Let $y\in \Fix (f )$ be such that $y\neq \bar y$. Up to extracting we may assume that $\ell_{y_{\lambda_n}-y}$ is weakly converging. Using lemma \ref{weak}, it follows that $\ell_{y_{\lambda_n}-y} \rightharpoonup \ell_{\bar y - y}$.

We now pass to the limit to obtain that 

$$ \ell_{\bar y - y}(\bar y) = \lim_{n\to +\infty} \ell_{y_{\lambda_n}-y} (y_{\lambda_n}) \leqslant 0.$$
The rest of the proof is the same as in Theorem \ref{fp2}.
\end{proof}

\section{The Hilbert case}\label{Hilbert}

We assume now that $(H , \langle \cdot , \cdot \rangle)$ is a real  Hilbert space. We denote by $\|\cdot \|$ the induced norm. In this section, $C\subset H$ is a closed, convex and bounded set such that $0\in C$ and $f : C\to C$ is a $1$-Lipschitz map. We will use the same notations as in section \ref{euclide}. We state without more detailed proofs:

\begin{lm}
 The set $\Fix(f)$ of fixed points of $f$ is a closed, nonempty and convex set.
 \end{lm}

Indeed, the convexity is obtained as previously, the fact it is not empty is the content of \cite{Browder}.

\begin{lm}\label{monotonebis}
The function $Id - f : C\to H$ is monotone meaning that 
$$\forall (x,y)\in H \times H,\quad \big\langle \big(x-f(x) \big) -\big(y-f(y)\big), x-y\big\rangle \geqslant 0.$$
\end{lm}

\begin{lm}
The function $\lambda \mapsto \|y_\lambda\|$ is constant if $0\in \Fix(f)$ and increasing otherwise.
\end{lm}

\begin{Th}\label{hilbert}
The family $(y_\lambda)_{\lambda \in (0,1)}$ weakly converges as $\lambda \to 1$ to $x^*$ the orthogonal projection of $0$ on $C$.
\end{Th}

\begin{proof}
By Kirszbraun's Theorem, we may extend $f$ to a $1$-Lipschitz map from $H$ to itself, that we still denote by $f$. Note that the previous lemmas remain true for this extension of $f$ to $H$. Recall that a closed bounded set is weakly compact.

Let $\lambda_n\to 1$ be such that $y_{\lambda_n}\rightharpoonup y^*$.

We first prove that $y^*$ is a fixed point of $f$ following \cite{Browder}. We apply lemma \ref{monotonebis} to $x= y^*+tv$ for any vector $v\in H$, $t\in \R$ and $y = y_{\lambda_n}$ to obtain
$$ \big\langle y^*+tv-f(y^*+tv)  +\big(\frac{1}{\lambda}_{n}-1\big)y_{\lambda_n}, tv+y^*-y_{\lambda_n}\big\rangle \geqslant 0.$$

As $n\to +\infty$, the left hand side of the bracket strongly converges to $y^*+tv-f(y^*+tv) $ (because the $y_{\lambda_n}$ are bounded) and the right hand side weakly converges to $tv$. Hence we obtain, passing to the limit that
$$\forall t\in \R, \ \forall v\in H, \quad t\langle y^*+tv-f(y^*+tv) , v\rangle \geqslant 0.$$

Letting $t>0$ going to $0$, we discover that $\langle y^* - f(y^*) , v\rangle \geqslant 0$ and letting $t<0$ going to $0$,  that $\langle y^* - f(y^*) , v\rangle \leqslant 0$. In conclusion
$$\forall v \in H , \quad \langle y^* - f(y^*),v\rangle = 0,$$
which establishes that $y^* $ is a fixed point of $f$.

We now notice that  by weak convergence,
\begin{equation}\label{infe}
\|y_{\lambda_n}\| \cdot \|y^* \| \geqslant \langle y_{\lambda_n},y^* \rangle  \to \|y^*\|^2.
\end{equation}

Hence $\|y^*\| \leqslant \lim\limits_{n\to +\infty} \|y_{\lambda_n}\|$. Applying now lemma \ref{monotonebis} to $x= y_{\lambda_n}$ and $y\in \Fix(f)$ it follows that 
$\langle y_{\lambda_n} , y_{\lambda_n}-y\rangle \leqslant 0$ for all $n$. Then, using \eqref{infe} and by weak convergence we get
$$\langle y^* , y^*-y \rangle = \|y^*\|^2 - \langle y^* , y \rangle \leqslant \lim_{n\to +\infty} \|y_{\lambda_n}\|^2 - \lim_{n\to +\infty} \langle y_{\lambda_n}, y \rangle \leqslant 0.$$

Hence $y^* $ is again the orthogonal projection of $0$ on $\Fix(f)$.

\end{proof}

\section{A concluding remark}

In all of the versions of our result, the point $0$ plays a distinguished role. This is arbitrary and can be dropped. In any of the contexts of Theorem \ref{fp1}, \ref{fixe1}, \ref{fp2}, \ref{fp3} or \ref{hilbert}, let us not assume anymore that $0\in C$. If $x\in C$ and $\lambda\in (0,1)$, the map $f_{\lambda,x} : y\mapsto \lambda f(y)+(1-\lambda) x$ is well defined and a contraction. If $y_{\lambda,x}$ is its unique fixed point, our theorem yields that as $\lambda\to 1$,  $y_{\lambda,x}$ converges to a fixed point of $f$ that we denote by $y_x\in \Fix(f)$. It is characterized by the property:
$$\forall z\in \Fix(f)\setminus\{y_x\}, \quad \ell_{y_x-z}(y_x-x) \leqslant 0.$$
If $(x,x')\in C\times C$ are such that $y_x\neq y_{x'}$, by two symmetric applications of the previous property, it follows that
$$ \ell_{y_x-y_{x'}}(y_x-x)-\ell_{y_x-y_{x'}}(y_{x'}-x') = \ell_{y_x-y_{x'}}(y_x-y_{x'})-\ell_{y_x-y_{x'}}(x-x')\leqslant 0.$$
Hence
$$\|y_x-y_{x'}\| \leqslant \ell_{y_x-y_{x'}}(x-x')\leqslant \|x-x'\|.$$
It follows that the map $x\mapsto y_x$ is $1$-Lipschitz. As it is obviously the identity on $\Fix(f)$ it is a continuous retraction of $C$ on $\Fix(f)$\footnote{Note that such a retraction in the cases where $C$ is compact can be obtained by a fixed point argument applied to the mapping $\rho \mapsto f\circ \rho$.}.
\appendix

\section{An example where $(x_\lambda)_{\lambda \in (0,1)}$ diverges}\label{A}

We consider here $(\R^2, \|\cdot \|_1)$ where the norm is defined by $\|(x,y)\|_1 = |x| + |y|$. The convex compact set we are interested in is the triangle
$$\T = \left\{  (x,y) \in \R^2 , \ \ -\frac12 \leqslant y \leqslant -|x|+\frac12 \right\}.$$
Of course, $(0,0)\in \T$. If $\alpha\in (0,1)$, we aim at constructing a function of the form $f(x,y) = \big( x+\e(y), \alpha \big(y+\frac12 \big) - \frac12 \big) $ where $\e : \big[-\frac12,\frac12\big] \to \R$ is a function to be determined verifying $\e\big(-\frac12\big)=0$. It will follow that the bottom side of $\T$ will be made of fixed points of $f$.

As we want $f$ to take values in $\T$, we have to check that for $y\in [-\frac12, \frac12]$, if $x\in \big[y-\frac12,\frac12 - y \big]$ then 
$-\frac12 \leqslant \alpha \big(y+\frac12 \big) - \frac12  \leqslant -|x+\e(y)|+\frac12$.
This is verified if and only if
\begin{equation}\label{condition1}
\forall y\in  \Big[-\frac12, \frac12 \Big],\quad |\e(y) | \leqslant (1-\alpha)\Big(y+\frac12\Big).
\end{equation}

The condition that $f$ is $1$-Lipschitz is realized if for all $(x,y)$ and $(x',y')$ in $\T$
$$\|f(x,y) - f(x',y')\|_1 =| x-x'+\e(y)-\e(y') | +  \alpha|y-y'| \leqslant |x-x'| + |y-y'|.$$ 

This is true if we take $\e$ to be $(1-\alpha)$-Lipschitz. Note that if this is the case and $\e\big(-\frac12\big)=0$ then \eqref{condition1} is verified.

We now assume  $\e$ is $(1-\alpha)$-Lipschitz. If $\lambda \in (0,1)$ remember that $f_\lambda$ is defined by $f_\lambda(x,y) = f(\lambda x, \lambda y)$. Denoting by $X_\lambda = (x_\lambda, y_\lambda)\in \T$ the unique fixed point of $f_\lambda$, an explicit computation gives 
$$(x_\lambda,y_\lambda) = \left( \frac{1}{1-\lambda}\varepsilon\Big(\frac{\lambda(\alpha-1)}{2(1-\alpha \lambda)} \Big) , \frac{\alpha -1}{2(1-\alpha \lambda)} \right).$$

By setting $g(1-\lambda) = \frac{\lambda(\alpha-1)}{2(1-\alpha \lambda)} $, one checks that 
$$g^{-1}(\mu) = \frac{(\alpha-1)(1+2\mu)}{\alpha-1-2\alpha \mu}.$$
Hence $g$ is an increasing bi-Lipschitz map from $[0,1]$ to $[-1/2,0]$.
We now set $h : \R \to \R$ defined by $h(x) = x\sin\big(\ln(x)\big)$ for $x\neq 0$ and $h(0)=0$ that is a Lipschitz function. It follows that for $\varepsilon_0>0$ small enough, the function $\varepsilon = \varepsilon_0 h \circ g^{-1} $ is $(1-\alpha)$-Lipschitz on $\big[-\frac12 , 0\big]$ and verifies $\e\big(-\frac12\big) = 0$. We extend it by $\e(y) = \e(0)$ if $y\in \big[0, \frac12 \big]$. 

For the function $f$ obtained this way, we find that $$(x_\lambda , y_\lambda ) =\big(\sin\big(\ln(1-\lambda)\big), \frac{1}{\lambda}g(1-\lambda)\big).$$
It does not converge as $\lambda \to 1$. 

\section{An example where $\Fix (f)$ is not convex}

The setting here is $(\R^2, \|\cdot \|_\infty)$ with $\|(x,y)\|_\infty = \max (|x|, |y|)$. Let $g : \R \to \R$ be a $\frac12$-Lipschitz map and $f : (x,y) \mapsto \big(x,g(x)\big)$ be the vertical projection on the graph of $g$. Then $f$ is $1$-Lipschitz. Indeed, if $\big((x,y),(x',y')\big) \in (\R^2)^2$,
$$ \|f(x,y)-f(x',y')\|_\infty = \max(|x-x'|, |g(x)-g(x')|) = |x-x'| \leqslant \|(x,y)-(x',y')\|_\infty.$$
Of course, the set of fixed points of $f$ is the graph of $g$ which is in general not convex.

Let now $R= [-1,1]\times [-M,M] \subset \R^2$ with $M = \max(|g(x)|, \ \ x\in [-1,1])$. Then by still denoting the restriction of $f$ to $R$ by $f$, we obtain a map $f : R\to R$. If $\lambda \in (0,1)$, the  point $X_\lambda$ such that $\lambda f(X_\lambda) = X_\lambda$ is $\big(0,\lambda f(0)\big)$ and obviously $X_\lambda \to \big(0,f(0)\big)$ as $\lambda \to 1$.

The norm $\|\cdot \|_\infty$ is not $C^1$ on $\R^2 \setminus \{(0,0)\}$. Let us consider $\| \cdot \|$ a norm obtained from $\|\cdot \|_\infty$ by rounding off the corners of the unit ball. It then verifies that $\|X\|\geqslant \|X\|_\infty$ for all $X\in \R^2$. Let us assume $\| \cdot \|$ is close enough to $\|\cdot \|_\infty$ in the sense that $\|(1,y)\| = 1$ is $y\in [-1/2,1/2]$. Then for all  pairs $(x,x')\in \R^2$, $\big\|\big(x,g(x)\big) - \big(x',g(x')\big)\big\| = \big\|\big(x,g(x)\big) - \big(x',g(x')\big)\big\|_\infty = |x-x'|$. It follows that for all $(x,x',y,y')\in \R^4$,
\begin{multline*}
 \|f(x,y)-f(x',y')\| =\big\|\big(x,g(x)\big) - \big(x',g(x')\big)\big\| \\
 = |x-x'| \leqslant \|(x,y)-(x',y')\|_\infty \leqslant \|(x,y)-(x',y')\|.
\end{multline*}
 Hence $f$ is still $1$-Lipschitz for $\|\cdot \|$.

\bibliography{fixe}
\bibliographystyle{siam}

\end{document}